\documentclass[10pt]{article}
\newtheorem{prelemmaa}{{\bf LEMMA}}

\newtheorem{prelem}{{\bf THEOREM}}

\newtheorem{preque}{{\bf QUESTION}}

\newtheorem{theorem}{THEOREM}

\newtheorem{prelemma}{LEMMA}

\newenvironment{lemma}{\begin{prelemma}{\hspace{-0.5
               em}}}{\end{prelemma}}
\newtheorem{preproof}{{\bf PROOF.}}

\newenvironment{proof}[1]{\begin{preproof}{\rm
               #1}\hfill{\rule[-0.5mm]{2mm}{2mm}}}{\end{preproof}}

\newtheorem{preproposition}{{PROPOSITION}}

\newenvironment{proposition}{\begin{preproposition}{\hspace{-0.5
                em}}}{\end{preproposition}}
\newtheorem{preremark}{REMARK}

\newtheorem{precorollary}{{COROLLARY}}

\newtheorem{predefinition}{DEFINITION}

\newtheorem{preexample}{EXAMPLE}

\newtheorem{preconjecture}{{CONJECTURE}}

\newtheorem{pretheo}{{\bf THEOREM}}

\def\newpic#1{}
\date{}
\begin{document}
\title{\bf Smallest defining sets of super-simple $2-(v,4,1)$ directed designs}
\author{
{\sc  Nasrin Soltankhah\footnote{ Corresponding  author. E-mail
address: soltan@alzahra.ac.ir} AND Farzane Amirzade}
 \\[5mm]
Department of Mathematics\\ Alzahra University  \\
Vanak Square 19834 \ Tehran, I.R. Iran }
%
\maketitle
\begin{abstract}
A $2-(v,k,\lambda)$ directed design (or simply a
$2-(v,k,\lambda)DD$) is super-simple if its underlying
$2-(v,k,2\lambda)BIBD$ is super-simple, that is, any two blocks of
the $BIBD$ intersect in at most two points. A $2-(v,k,\lambda)DD$ is
simple if its underlying  $2-(v,k,2\lambda)BIBD$ is simple, that is,
it has no repeated  blocks.

 A set of blocks which is a subset of a
unique $2-(v,k,\lambda)DD$ is said to be a defining set of the
directed design. A smallest defining set, is a defining set which
has smallest cardinality. In this paper simultaneously  we show that
the necessary and sufficient condition for the existence of a
super-simple $2-(v,4,1)DD$ is $v\equiv1\ ({\rm mod}\ 3)$ and for
these values except $v=7$ , there exists a super-simple
$2-(v,4,1)DD$ whose smallest defining sets have at least a half of
the blocks. And also for all $\epsilon > 0$ there exists
$v_0(\epsilon)$ such that for all admissible $v>v_0$ there exists a
$2-(v,4,1)DD$ whose smallest defining sets have at least
$(\frac{5}{8}-\frac{c}{v})\mid \mathcal{B}\mid$ blocks, for suitable
positive constant c.
 \end{abstract}
%
\hspace*{-2.7mm} {\bf KEYWORDS:} {  \sf super simple directed designs, smallest
defining set }
%
\section{Introduction} 
\setcounter{preproposition}{0} \setcounter{precorollary}{0}
\setcounter{prelemma}{0} \setcounter{preexample}{0}
A $t-(v,k,\lambda)$ directed design (or simply a
$t-(v,k,\lambda)DD$) is a pair $(V,\mathcal{B})$, where $V$ is
$v$-set, and $\mathcal{B}$ is a collection of  ordered $\it
k-$tuples of distinct elements of $V$ (called blocks), such that
each ordered $t$-tuple of distinct elements of $V$ appears in
precisely $\lambda$ blocks. We say that a $t$-tuple appears in a
$\it k$-tuple, if its components appear in that $k$-tuple as a set,
and they appear with the same order. For example the following
4-tuples form a $2-(10,4,1)DD, \mathcal{D}$.
$$
\begin{array}{llllllll}
0132 & 0467 & 0589 & 2680 & 2791 & 7843 & 6942 & 4510 \\
1498 & 3861 & 3970 & 8752 & 9653 & 2354 & 1576.
\end{array}
$$
Here, for example, the 4-tuple 0132 contains the ordered pairs 01,
03, 02, 13, 12, and 32.

 A $(v,k,t)$ directed trade of volume $s$ consists of two
disjoint collections
 $T_1$ and $T_2$, each of
$s$ blocks, such that every $t$-tuple of distinct elements of $V$ is
covered by precisely the same number of blocks of $T_1$ as of $T_2$.
Such a directed trade is usually denoted by $T=T_1-T_2$. Blocks in
$T_1(T_2)$ are called the positive (respectively, negative) blocks
of T. If ${\mathcal{D}}=(V,\mathcal{B})$ is a directed design, and
if $T_1\subseteq \mathcal{B}$, we say that $\mathcal{D}$ contains
the directed trade T. For example the $2-(10,4,1)DD, \mathcal{D}$
above contains the following directed trade:
$$
\begin{array}{ccccccc}
\frac{\ \sc T_1\ }{\ }  &      & \frac{\ \sc T_2\ }{\ }  \\
0132                    &      & 1032                \\
4510                    &      & 4501
\end{array}
$$
Given a $t-(v,k,\lambda)$ directed design $\mathcal{D}$, a subset of
the blocks of $\mathcal{D}$ that occurs in no other
$t-(v,k,\lambda)$ directed design is called a defining set of
$\mathcal{D}$. A defining set for which no other defining set has a
smaller cardinality, is called a smallest defining set. For example
the set of blocks $S=\{0132,2354,9653,8752,2791,2680,1498,1576\}$ is
a defining set of the $2-(10,4,1)DD$ above. But the set of blocks
$R=\{0132,2354,9653,8752, 2791,\\2680,1498,0467\}$ can be completed
to a $2-(10,4,1)DD$ in two ways: by adjoining either
$\{4510,1576,0589,3861,3970,7843,6942\}$ or
$\{4150,5176,0589,\\3861,3970,7843,6942\}$.

Defining sets for directed designs (as suggested by A. P. Street in
$\cite{A.P.Street}$) are strongly related to trades. This relation
is illustrated by the following result.
\begin{proposition}~{\rm \cite{MMS}} Let $\mathcal{D}=(V,\mathcal{B})$ be a $t-(v,k,\lambda)$ directed design and let $S\subseteq
\mathcal{B}$, then $S$ is a defining set of $\mathcal{D}$ if only if
$S$ contains a block of every $(v,k,t)$ directed trade $T=T_1-T_2$
such that $T$ is contained in $\mathcal{D}$.
\end{proposition}
Each defining set of a $t-(v,k,\lambda)DD, \mathcal{D}$ contains at
least one block in every trade in $\mathcal{D}$. In particular, if
$\mathcal{D}$ contains $m$ mutually disjoint directed trades then
the smallest defining set of $\mathcal{D}$ must contain at least $m$
blocks.

The concept of directed trades and defining sets for directed
designs were investigated in articles $\cite{MMS, trade}$.

A $2-(v,k,\lambda)$ directed design (or simply a
$2-(v,k,\lambda)DD$) is super-simple if its underlying  $2-(v,k,2\lambda)BIBD$ is super-simple, that is,
any two blocks of the $BIBD$ intersect in at most two points. A $2-(v,k,\lambda)DD$ is simple if its underlying  $2-(v,k,2\lambda)BIBD$ is simple, that is,
it has no repeated  blocks.

The concept of super-simple designs was introduced by Gronau and
Mullin $\cite{gronau1992}$. The existence of super-simple designs is
an interesting problem by itself, but there are also some useful
applications.

There are known results for the existence of super simple designs,
especially for existence of super simple $2-(v,k,\lambda)BIBDs$.
When $k=4$  the necessary conditions for super simple
 $2-(v,k,\lambda)BIBDs$ are known to be sufficient for $2\leq \lambda \leq 6$ with few
 possible exceptions. These known results can be found in articles
 $\cite{Adams, Cao, Chen1995, Chen1996, Chen2005, gronau2007, Khodkar}$.

In $\cite{MMS}$, Mahmoodian, Soltankhah and Street have proved that
if $\mathcal{D}$ be a $2-(v,3,1)DD$ then a defining set of
$\mathcal{D}$ has at least $\frac{v}{2}$ blocks. In $\cite{Quinn}$
Grannel,  Griggs and Quinn have shown that for each admissible value
of $v$, there exists a $2-(v,3,1)DD$ and a simple $2-(v,3,1)DD$
whose
 smallest defining sets have at least $\frac{\mid
\mathcal{B}\mid}{2}$ blocks. And they have also proved that, for all
$\epsilon>0$ and all sufficiently large admissible $v$, there exists
a $2-(v,3,1)DD$ whose smallest defining sets have at least
$(\frac{2}{3}-\epsilon )\mid \mathcal{B}\mid$ blocks.

In this paper we show that for all admissible values, there is a
 super-simple $2-(v,4,1)DD$ (except
$v=7$) whose smallest defining sets have at least a half of the
blocks.

In other words we are interested in the quantity
\begin{center}
$\large f=\frac{{\rm number}\ {\rm of}\ 4-{\rm tuples}\ {\rm in}\
{\rm a}\ {\rm smallest}\ {\rm defining}\ {\rm set}\ {\rm in}\
\mathcal{D}} {{\rm number}\ {\rm of}\ 4-{\rm tuples}\ {\rm in}\
\mathcal{D}}.$
\end{center}
And we show that, there exists a super-simple $2-(v,4,1)DD$,
$\mathcal{D}$ with $f\geq\frac{1}{2}$.

The proofs in this paper use various types of combinatorial objects.
The definitions of these objects are either given in the paper or
can be found in the references.

Several proofs depend on the following result, which involves
pairwise balanced designs (PBDs) and is a special case of a result
(the Replacement Lemma $\cite{MMS}$) that is used in several earlier
papers on directed designs.
\begin{lemma}
If there exist a $2-(v,K,1)$ design and a super-simple $(k,4,1)DD$
for each $k\in K$, then there exists a super-simple $2-(v,4,1)DD$.
\end{lemma}
\begin{proof}
{Replacing each block of the $2-(v,K,1)$ design with a copy of a
super-simple $2-(k,4,1)DD$ with point set the points of that block
gives a super-simple $2-(v,4,1)DD$.}
\end{proof}
A lower bound for $\it f$ for the super-simple $2-(v,4,1)DD$ constructed in Lemma
1.1 can be calculated from lower bounds for $\it f$ for the various
super-simple $2-(k,4,1)DD$s. In particular, if there is a constant c such that
each of the super-simple $2-(k,4,1)DD$s has $f\geq c$, then the resulting super-simple
$2-(v,4,1)DD$ also has $f\geq c$.
%
\section{Main result}  
A necessary and sufficient condition for the existence of a
$2-(v,4,1)DD$ is $v\equiv1\ ({\rm mod}\ 3)$ $\cite{DBIBD}$.

In this section simultaneously  we show that the necessary and
sufficient condition for the existence of a super-simple $2-(v,4,1)DD$
is $v\equiv1\ ({\rm mod}\ 3)$ and for these values except $v=7$,
there exists a super-simple $2-(v,4,1)DD$ with $f\geq\frac{1}{2}$.

In
$\cite{soltan}$, Soltankhah and Mahmoodian have shown that, up to
isomorphism, there exist two $2-(7,4,1)DDs$  that they are super-simple and each of them
has smallest defining set of cardinality 2. Consequently, if $v=7$
then $f=\frac{2}{7}$.

In proofs we handle with a special type of directed trade, named a
cyclical trade, as follows.

Let $T=T_1-T_2$ be a $(v,4,2)$ directed trade of volume $s$, where
$T_1$ contains blocks $b_1,\dots,b_s$ such that each pair of
consecutive 4-tuples (blocks) of $T_1$, $b_ib_{i+1}$ $i=1,\dots,s\
({\rm mod}\ s)$ is a trade of volume 2. Therefore if a directed
design $\mathcal{D}$ contains $T_1$, then any defining set for
$\mathcal{D}$ must contain at least $[\frac{s+1}{2}]$ blocks of
$T_1$.

The following
4-tuples form a super-simple $2-(10,4,1)DD$.
$$
\begin{tabular}{|c|c|c|c|c|c|c|} \hline
0467 & 0589 & 2680 & 2791 & 7843 & 6942 & 0132 \\
1576 & 1498 & 3861 & 3970 & 8752 & 9653 & 4510 \\
     &      &      &      &      &      & 2354 \\ \hline
\end{tabular}
$$
Each of the first six columns above contains a trade, hence any
defining set for this directed design must contain at least one
4-tuple of each column. The last column above is a cyclical trade of
volume 3. Hence any defining set for this directed design must
contain at least two 4-tuples of the last column. We can use integer
programming problem $\cite{STS}$ to find a smallest defining set for
this directed design and
$S=\{0132,2354,9653,8752,2791,2680,\\1498,1576\}$ is a defining set of
size 8. Therefore this super-simple $2-(10,4,1)DD$ has
$f=\frac{8}{15}>\frac{1}{2}$.

The following base blocks by $(+1\ {\rm mod}\ 13)$ form a super-simple
$2-(13,4,1)DD$ with $f\geq\frac{1}{2}$.
$$0\ 1\ 3\ 9\ ,\ 1\ 0\ 11\ 5$$
The following 4-tuples form a super-simple $2-(16,4,1)DD$.

$$
\begin{tabular}{|c|c|c|c|c|} \hline
1 3 2 4     & 1 5 10 14  & 1 6 12 13  & 3 8 10 15  & 1 8 11 16  \\
16 14 3 1   & 13 9 5 1   & 15 10 6 1  & 14 10 8 9  & 16 11 10 2 \\
            &            &            &            &            \\
5 6 7 8     & 2 6 11 15  & 2 5 9 16   & 4 7 11 14  & 2 7 10 13  \\
16 13 7 6   & 14 11 6 5  & 15 5 2 3   & 10 7 4 5   & 8 7 2 1    \\
            &            &            &            &            \\
9 10 12 11  & 3 7 12 16  & 2 8 12 14  & 1 7 9 15   & 3 6 9 14   \\
13 12 10 3  & 16 12 8 5  & 15 14 12 7 & 11 9 7 3   & 12 9 6 2   \\
            &            &            &            &            \\
13 14 15 16 & 4 9 8 13   & 3 5 11 13  & 4 6 10 16  & 5 4 12 15  \\
14 13 4 2   & 16 15 9 4  & 15 13 11 8 & 8 6 4 3    & 11 12 4 1  \\
\hline
\end{tabular}
$$

\noindent This super-simple $2-(16,4,1)DD$ contains 20 disjoint
directed trades of volume 2. Hence any defining set of this
super-simple directed design must contain at least one 4-tuple of
each trade. So for this super-simple $2-(16,4,1)DD$ we have
$f\geq\frac{20}{40}=\frac{1}{2}$.

In our future results we need some super-simple directed group
divisible designs with block size four (super-simple $4-DGDDs$). We
may use  a super-simple $4-DGDD$ of type $(3^4)$, a super-simple
$4-DGDD$ of type $(3^5)$, a super-simple $4-DGDD$ of type $(3^6)$
and a super-simple $4-DGDD$ of type $(2^4)$. Now we can construct a
super-simple $4-DGDD$ of type $(3^4)$ with groups $\{0,4,8\},\
\{1,5,9\},\ \{2,6,10\},\ \{3,7,11\}$, and with the following blocks.

$$
\begin{tabular}{|c|c|c|c|c|c|c|c|} \hline
4 3 10 9 & 9 2 3 8  & 2 1 4 11  & 11 10 5 4 & 7 6 0 1  \\
9 10 0 7 & 11 0 2 9 & 3 4 1 6   & 4 5 7 2   & 6 7 9 4  \\
1 0 10 3 & 3 2 0 5  &           &           &          \\
         &          & 8 1 2 7   & 5 0 6 11  & 7 8 5 10 \\
         &          & 10 11 1 8 & 8 9 11 6  & 6 5 8 3  \\
\hline
\end{tabular}
$$

\noindent Each of the first two columns above is a cyclical trade of
volume 3, hence any defining set of this super-simple directed group
divisible design must contain at least two 4-tuples from each of the
first two columns and 4-tuples in the last three columns above form
six disjoint directed trades of volume 2 and one 4-tuple from each
of them must be in any defining set. Therefore any defining set of
this super-simple $4-DGDD$ must contain at least 10 blocks, so it
has $f\geq\frac{10}{18}>\frac{1}{2}$.

A super-simple $4-DGDD$ of type $(3^5)$ can be constructed with groups
$$\{0,5,10\},\ \{1,6,11\},\ \{2,7,12\},\ \{3,8,13\},\ \{4,9,14\}$$ and
with the following base blocks by $(+1\ {\rm mod}\ 15)$.
$$1\ 0\ 3\ 7\ ,\ 0\ 1\ 13\ 9$$
This super-simple directed group divisible design has 15 disjoint
directed trades of volume 2 and one 4-tuple from each of them must
be in any defining set. Therefore any defining set of this super-simple $4-DGDD$ must
 contain at least 15 blocks, so it has
$f\geq\frac{15}{30}=\frac{1}{2}$.

A super-simple $4-DGDD$ of type $(3^6)$ can be constructed with groups
$$\{0,6,12\},\ \{1,7,13\},\ \{2,8,14\},\ \{3,9,15\},\ \{4,10,16\},\ \{5,11,17\}$$ and
with the following base blocks by $(+2\ {\rm mod}\ 18)$.

$$
\begin{array}{l}
5\ 10\ 2\ 0\ \ ,\ 14\ 1\ 10\ 5 \\
15\ 0\ 4\ 11\  ,\ 11\ 4\ 3\ 1\ \\
0\ 1\ 3\ 2
\end{array}
$$

This super-simple directed group divisible design has 18 disjoint
directed trades of volume 2 in the first two rows above and the last
row is a cyclical trade of volume 9. So any defining set for this
super-simple $4-DGDD$ must contain at least five 4-tuples of the last
row. Hence any defining set must contain at least $18+5=23$ blocks,
so for this super-simple $4-DGDD$, we have
$f\geq\frac{23}{45}>\frac{1}{2}$.

A super-simple $4-DGDD$ of type $(2^4)$  with $f\geq\frac{1}{2}$ can
be constructed with groups $\{1,2\},\ \{3,4\},\ $ $\{5,6\},\
\{7,8\}$ and with the following blocks

$$
\begin{array}{l}
4\ 1\ 6\ 7\ ,\ 8\ 6\ 1\ 3 \\
6\ 8\ 2\ 4\ ,\ 7\ 5\ 4\ 2 \\
5\ 7\ 3\ 1\ ,\ 2\ 3\ 7\ 6 \\
1\ 4\ 5\ 8\ ,\ 3\ 2\ 8\ 5.
\end{array}
$$

Our principal tool is to apply Wilson's Fundamental construction(weighting), that is described in the following
lemma.
\begin{lemma}
If there is a $\{K\}-GDD$ of type ${g_1}^{u_1}{g_2}^{u_2}\dots {g_N}^{u_N}$, there are  super-simple $2-(\alpha g_i+1,4,1)DD$s for each
$i,\ i=1,2,\dots,N$ and there are super-simple $4-DGDD$s of type $\alpha ^k$ for each $k\in K$ then, there exists a super-simple $2-(\alpha\sum_{i=1}^{N}g_iu_i+1,4,1)DD$.
\end{lemma}
\begin{proof}
{Let $(G,\mathcal{B})$ be a group divisible design with
element set U, blocks of size $k\in K$ and groups of size $g_1,g_2,\dots,$ and $g_N$. We form a super-simple $2-(\alpha\sum_{i=1}^{N}g_iu_i+1,4,1)DD$ on the element set $U\times
Z_{\alpha} \bigcup\{\infty\}$. Give weight $\alpha$ to all of points. That is replace each point
$x\in U$ with $\alpha$ new points $\{x_1,x_2,\dots,x_{\alpha}\}$. Now replace each block $b\in \mathcal{B}$ of size $k\in K$ with a  super-simple $4-DGDD$ of type $\alpha^k$ such that its groups are $\{x_1,x_2,\dots,x_{\alpha}:x\in U\}$ to get a super-simple $4-DGDD$ of type ${\alpha g_1}^{u_1}{\alpha g_2}^{u_2}\dots {\alpha g_N}^{u_N}$. Finally fill in the holes with a new point  $\infty$, using  super-simple
$2-(\alpha g_i+1,4,1)DD$s for $i=1,2,\dots,N$.
}
\end{proof}
In particular, if there is a constant c such that each of the super-simple $4-DGDD$ of type $\alpha^k$ and the super-simple
$2-(\alpha g_i+1,4,1)DD$s has $f\geq c$, then the resulting super-simple $2-(v,4,1)DD$ in Lemma 2.2 also has $f\geq c$.
\begin{theorem}
For all $v\equiv1\ (mod\ 3)$ except $v=7$ there exists a super-simple $2-(v,4,1)DD$ with $f\geq\frac{1}{2}$.
\end{theorem}
\begin{proof}
{For all $n$ except $n\in A = \{7, 8, 9, 10, 11, 12, 14, 15,
18,19,23\}$ there exists a $PBD(n, \{4, 5, 6\})$ $\cite{mullin}$. We
can delete one point from this $PBD$ to form a $\{4, 5, 6\}-GDD$ of
order $n-1$ of type $3^{a}4^{b}5^{c}$, where a, b, c are
non-negative integers. We then apply lemma 2.2 to this $GDD$ using a
weight of 3 to get a super-simple $4-DGDD$ of type
$9^{a}12^{b}15^{c}$ with $f\geq\frac{1}{2}$. Finally fill in the
holes with a new point  $\infty$, using  super-simple $2-(m,4,1)DD$s
for $m=10,\ 13,\ 16$ with $f\geq\frac{1}{2}$.

Now for the remaining values $v\in\{19,22,25,28,31,34,40,43,52,55,67\}$ we
construct a super-simple $2-(v,4,1)DD$ with $f\geq\frac{1}{2}$ as follows.

\noindent$\bullet\ \ \ v=19$: the following 4-tuples form a
super-simple $2-(19,4,1)DD$.
\begin{center}
\begin{tabular}{|c|c|c|c|c|c|c|} \hline
3 6 4 15   & 2 7 4 17   & 10 3 11 13 & 11 16 5 12 & 14 12 17 8  \\
6 3 16 1   & 5 4 7 16   & 9 3 10 14  & 13 15 12 5 & 6 17 12 13  \\
           &            &            &            & 5 8 17 6    \\
6 11 0 7   & 17 9 4 5   & 3 9 12 7   & 10 12 15 6 &             \\
11 6 8 2   & 16 9 17 0  & 1 8 7 12   & 12 10 9 2  & 2 1 14 6    \\
           &            &            &            & 13 4 1 2    \\
4 0 18 12  & 13 16 7 11 & 8 1 9 11   & 2 9 13 8   & 18 14 1 4   \\
18 0 2 11  & 16 13 6 14 & 9 1 15 16  & 15 13 9 18 &             \\
           &            &            &            & 0 5 1 13    \\
18 7 9 6   & 7 14 18 13 & 8 18 16 15 & 8 4 13 10  & 2 5 0 15    \\
0 4 6 9    & 17 14 7 15 & 2 12 16 18 & 16 4 8 3   & 15 0 8 14   \\
           &            &            &            & 7 10 8 0    \\
12 1 3 0   & 5 14 11 9  & 15 11 17 1 & 14 0 10 16 & 10 7 1 5    \\
13 0 3 17  & 12 4 11 14 & 11 15 10 4 &            &             \\
           &            &            &            &             \\
1 18 10 17 & 14 5 3 2   & 17 11 3 18 & 17 16 2 10 &             \\
6 5 10 18  & 15 7 2 3   & 18 3 8 5   &            &             \\
\hline
\end{tabular}
\end{center}
This  super-simple $2-(19,4,1)DD$ has 22 disjoint directed trades of
volume 2 in the first four columns above and the last column has
three cyclical trades of volume 3, 3 and 5, respectively. So any
defining set for this $DD$ must contain at least two 4-tuples from
each of the cyclical trades of volume 3 and three 4-tuples of the
cyclical trade of volume 5. Hence any defining set must contain at
least $22+2\times2+3=29$ blocks, so for this super-simple
$2-(19,4,1)DD$, we have$f\geq\frac{29}{57}>\frac{1}{2}$.

\noindent$\bullet\ \ \ v=22$: the following base blocks by $(-,+1\
{\rm mod}\ 11)$ form a super-simple $2-(22,4,1)DD$.
$$
\begin{array}{l}
(0 , 0) (0 , 3) (0 , 9) (0 , 10)\  ,\ (1 , 0) (1 , 4) (0 , 9) (0 , 3)\\
(0 , 0) (1 , 0) (1 , 7) (1 , 2)\ \  ,\ (1 , 0) (0 , 0) (1 , 9) (1 , 10)
\\(0 , 0) (0 , 2) (1 , 5) (1 , 8)\ \ ,\ (1 , 0) (1 , 5) (0 , 2) (0 , 6)\\
(1 , 7) (0 , 3) (0 , 0) (1 , 4)
\end{array}
$$
This super-simple $2-(22,4,1)DD$ has 33 disjoint directed trades of
volume 2 in the first three rows above and the last row is a
cyclical trade of volume 11. So any defining set for this $DD$ must
contain at least six 4-tuples of the last row. Hence any defining
set must contain at least $33+6=39$ blocks, so for this super-simple
$2-(22,4,1)DD$, we have $f\geq\frac{39}{77}>\frac{1}{2}$.

\noindent$\bullet\ \ \ v=25$: given base blocks in
$\cite{gronau1992}$ form a super-simple $2-(25,4,1)DD$ with
$f\geq\frac{1}{2}$.

\noindent$\bullet\ \ \ v=28$: the following 4-tuples form a
super-simple $2-(28,4,1)DD$.
\begin{center}
\begin{tabular}{|c|c|c|c|c|c|c|} \hline
1 2 14 16   & 1 8 22 23   & 2 3 19 22   & 6 8 9 14    & 4 7 14 15   \\
21 9 2 1    & 25 4 8 1    & 25 14 3 2   & 27 8 6 3    & 23 21 15 14 \\
            &             &             &             &             \\
12 14 20 24 & 23 22 9 3   & 13 14 19 26 & 3 6 13 17   & 2 6 21 23   \\
16 14 12 10 & 10 9 22 24  & 11 14 13 9  & 27 25 17 13 & 11 10 6 2   \\
            &             &             &             &             \\
6 7 10 12   & 8 4 21 24   & 7 13 8 11   & 7 16 25 27  & 10 11 23 25 \\
24 15 10 7  & 25 24 21 5  & 22 19 11 8  & 17 16 7 4   & 23 11 7 1   \\
            &             &             &             &             \\
2 8 10 15   & 1 3 24 25   & 2 7 18 24   & 16 17 20 22 & 1 7 9 17    \\
17 9 10 8   & 18 10 3 1   & 12 8 7 2    & 20 17 14 1  & 26 9 7 5    \\
            &             &             &             &             \\
6 15 24 27  & 0 3 10 14   & 9 11 21 26  & 14 17 21 25 & 5 7 19 20   \\
27 24 16 9  & 24 11 3 0   & 27 21 11 4  & 26 21 17 6  & 20 19 9 4   \\
            &             &             &             &             \\
4 3 9 16    & 1 6 20 26   & 5 6 22 25   & 3 12 21 27  & 13 16 23 24 \\
26 15 3 4   & 22 16 6 1   & 19 14 6 5   & 19 17 12 3  & 24 23 17 2  \\
            &             &             &             &             \\
0 2 12 17   & 1 4 11 12   & 2 4 25 26   & 11 17 19 24 & 3 5 11 15   \\
20 16 2 0   & 26 20 12 11 & 26 25 22 10 & 17 11 18 5  & 16 13 5 3   \\
            &             &             &             &             \\
4 6 18 19   & 0 5 24 26   & 3 8 18 20   & 18 15 17 23 & 0 4 20 23   \\
24 12 6 4   & 26 24 14 8  & 24 22 20 18 & 22 17 15 0  & 25 23 20 6  \\
            &             &             &             &             \\
0 8 19 25   & 5 8 12 16   & 2 11 20 27  & 4 13 22 27  & 18 14 4 0   \\
25 19 18 7  & 23 8 5 0    & 27 20 10 5  & 27 22 14 7  & 14 18 11 22 \\
            &             &             &             &             \\
8 17 26 27  & 3 7 23 26   & 0 1 13 15   & 10 13 20 21 & 9 15 20 25  \\
27 26 1 0   & 26 23 19 16 & 24 19 13 1  & 21 20 7 3   & 20 15 8 13  \\
            &             &             &             &             \\
15 18 9 6   & 2 5 9 13    & 6 0 11 16   & 12 15 22 26 & 13 7 0 6    \\
0 9 18 27   & 22 4 5 2    & 25 15 16 11 & 15 12 5 1   & 0 7 21 22   \\
            &             &             &             &             \\
16 15 19 21 & 12 13 18 25 & 1 10 19 27  & 5 14 23 27  & 5 4 10 17   \\
21 19 10 0  & 22 21 13 12 & 27 19 15 2  & 27 23 18 12 & 23 13 10 4  \\
            &             &             &             &             \\
1 5 18 21   & 10 16 18 26 & 25 12 9 0   &             &             \\
21 18 16 8  & 26 18 13 2  & 9 12 19 23  &             &             \\
\hline
\end{tabular}
\end{center}
This super-simple  $2-(28,4,1)DD$ contains 63 disjoint directed
trades. Hence, any defining set of this super simple directed design
must contain at least one 4-tuple of each trade. So for this
super-simple $2-(28,4,1)DD$ we have
$f\geq\frac{63}{126}=\frac{1}{2}$.

\noindent$\bullet\ \ \ v=31$: the following base blocks by $(+1\
{\rm mod}\ 31)$ form a super-simple $2-(31,4,1)DD$ with
$f\geq\frac{1}{2}$.
$$
\begin{array}{l}
2\ 0\ 4\ 1\ \ \ ,\ 1\ 4\ 19\ 9\\
4\ 0\ 20\ 13\ ,\ 13\ 20\ 8\ 30\ ,\ 14\ 28\ 8\ 20
\end{array}
$$
This super-simple $2-(31,4,1)DD$ has 31 disjoint directed trades of
volume 2 in the first row above and the second row is a cyclical
trade of volume 93. So any defining set for this $DD$ must contain
at least 47, 4-tuples of the cyclical trade. Hence any defining set
must contain at least $31+47=78$ blocks, so for this super-simple
$2-(31,4,1)DD$ we have $f\geq\frac{78}{155}>\frac{1}{2}$.

\noindent$\bullet\ \ \ v=34$: the following base blocks by $(+1\
{\rm mod}\ 34)$ form a super-simple
 $2-(34,4,1)DD$ with $f\geq\frac{1}{2}$.
$$
\begin{array}{l}
0\ 22\ 1\ 24\ ,\ 31\ 29\ 24\ 1 \\
25\ 1\ 19\ 0\ ,\ 19\ 1\ 15\ 22\ ,\ 23\ 15\ 1\ 20 \\
0\ 17\ 8\ 25
\end{array}
$$
This super-simple $2-(34,4,1)DD$ has 34 disjoint directed trades of
volume 2 in the first row above and the second row is a cyclical
trade of volume 102. So any defining set for this $DD$ must contain
at least 51, 4-tuples of the second row. The last row form a
cyclical trade of volume 34. Hence any defining set must contain at
least $34+51+17=102$ blocks, so for this super-simple $2-(34,4,1)DD$
we have $f\geq\frac{102}{204}\geq\frac{1}{2}$. $\ \ \ \ \ \ \ \ \ \
\ \ \ \ \ \ \ \ \ \ \ \ \ \ \ \ \ \ \ \ \ \ \ \ \ \ \ \ \ \ \ \ \ \
\ \ \ \ \ \ \ \ \ \ \ \ \ \ \ \ \ \ \ \ \ \ \ \ \ \ \ \ \ \ \ \ \ \
\ \ \ \ \ \ \ \ \ \ \ \ \ \ \ \ \ \ \ \ \ \ \ \ \ \ \ \ \ \ \ \ \ \
\ \ \ \ \ \ \ \ \ \ \ \ \ \ \ \ \ \ \ \ \ \ \ \ \ \ \ \ \ \ \ \ \ \
\ \ \ \ \ \ \ \ \ \ \ \ \ \ \ \ \ \ \ \ \ \ \ \ \ \ \ \ \ \ \ \ \ \
\ $

\noindent$\bullet\ \ \ v=40$: begin with a $4-GDD (5^4)$ $\cite{mullin}$. We apply
lemma 2.2 using a weight 2, using a super-simple $4-DGDD$ of type
$(2^4)$ with $f\geq\frac{1}{2}$, to get a super-simple $4-DGDD$ of
type $10^{4}$ with $f\geq\frac{1}{2}$. Now fill in the holes of size 10
with a super-simple $2-(10,4,1)DD$ with $f\geq\frac{1}{2}$. $\
 \ \ \ \ \ \ \ \ \ \ \ \ \ \ \
\ \ \ \ \ \ \ \ \ \ \ \ \ \ \ \ \ \ \ \ \ \ \ \ \ \ \ \ \ \ \ \ \ \
\ \ \ \ \ \ \ \ \ \ \ \ \ \ \ \ \ \ \ \ \ \ \ \ \ \ \ \ \ \ \ \ \ \
\ \ \ \ \ \ \ \ \ \ \ \ \ \ \ \ \ \ \ \ \ \ \ \ \ \ \ \ \ \ \ \ \ \
\ \ \ \ \ \ \ \ \ \ \ \ \ \ \ \ \ \ \ \ \ \ \ \ \ \ \ \ \ \ \ \ \ \
\ \ \ \ \ \ \ \ \ \ \ \ \ \ \ \ \ \ \ \ \ \ \ \ \ \ \ \ \ $

\noindent$\bullet\ \ \ v=43$: begin with a $4-GDD (2^7)$ $\cite{mullin}$. We apply
lemma 2.2 using a weight 3, using a super-simple $4-DGDD$ of type
$(3^4)$ with $f\geq\frac{5}{9}$, to get a super-simple $4-DGDD$ of
type $6^{7}$ with $f\geq\frac{1}{2}$. Finally fill in the hole with
a new point $\infty$, using a super-simple $2-(7,4,1)DD$. This yield
a super-simple $2-(43,4,1)DD$ with
$f\geq\frac{154}{301}>\frac{1}{2}$. $\ \ \ \ \ \ \ \ \ \ \ \ \ \ \ \
\ \ \ \ \ \ \ \ \ \ \ \ \ \ \ \ \ \ \ \ \ \ \ \ \ \ \ \ \ \ \ \ \ \
\ \ \ \ \ \ \ \ \ \ \ \ \ \ \ \ \ \ \ \ \ \ \ \ \ \ \ \ \ \ \ \ \ \
\ \ \ \ \ \ \ \ \ \ \ \ \ \ \ \ \ \ \ \ \ \ \ \ \ \ \ \ \ \ \ \ \ \
\ \ \ \ \ \ \ \ \ \ \ \ \ \ \ \ \ \ \ \ \ \ \ \ \ \ \ \ \ \ \ \ \ \
\ \ \ \ \ \ \ \ \ \ \ \ \ \ \ \ \ \ \ \ \ \ \ \ \ \ \ \ \ $

\noindent$\bullet\ \ \ v=52$: the following base blocks form a
super-simple $(52,4,1)DD$ with $f\geq\frac{1}{2}$.
$$
\begin{array}{l}
0\ 5\ 1\ 3\ \ \ \ ,\ 48\ 5\ 0\ 45\ \ \ \ \ \ \ (+1\ {\rm mod}\ 52)\\
30\ 6\ 0\ 13\ ,\ 13\ 0\ 23\ 42\ \ \ \ \ (+1\ {\rm mod}\ 52)\\
0\ 6\ 21\ 37\ ,\ 4\ 37\ 21\ 45\ \ \ \ \ (+1\ {\rm mod}\ 52)\\
8\ 38\ 20\ 0\ ,\ 34\ 9\ 0\ 20\ \ \ \ \ \ \ (+1\ {\rm mod}\ 52)\\
27\ 1\ 0\ 26\ ,\ 26\ 0\ 25\ 51\ \ \ \ \ (+2\ {\rm mod}\ 52)
\end{array}
$$
$\ \ \ \ \ \ \ \ \ \ \ \ \ \ \ \ \ \ \ \ \ \ \ \ \ \ \ \ \ \ \ \ \ \ \ \ \ \ \ \ \ \ \ \ \ \
\ \ \ \ \ \ \ \ \ \ \ \ \ \ \ \ \ \ \ \ \ \ \ \ \ \ \ \ \ \ \ \ \ \ \ \ \ \ \ \ \ \ \ \ \ \ \
\ \ \ \ \ \ \ \ \ \ \ \ \ \ \ \ \ \ \ \ \ \ \ \ \ \ \ \ \ \ \ \ \ \ \ \ \ \ \ \ \ \ \ \ \ \ \
\ \ \ \ \ \ \ \ \ \ \ \ \ \ \ \ \ \ \ \ \ \ \ \ \ \ \ \ \ \ \ \ \ \ \ \ \ \ \ \ \ $

\noindent$\bullet\ \ \ v=55$: a $RBIBD(16,4,1)$ has 5 parallel
classes. For two parallel classes add a new point to each of the
blocks in that parallel class. This yields a $\{4 , 5\}-GDD$ of type
$4^{4}2^{1}$. We apply lemma 2.2 using a weight 3, using a
super-simple $4-DGDD$ of type $(3^4)$ with $f\geq\frac{5}{9}$ and a
super-simple $4-DGDD$ of type $(3^5)$ with $f\geq\frac{1}{2}$, to
get a super-simple $4-DGDD$ of type $12^{4}6^{1}$ with
$f\geq\frac{1}{2}$. Finally fill in the holes with a new point
$\infty$, using a super-simple $(13,4,1)DD$ with $f\geq\frac{1}{2}$
and a super-simple $2-(7,4,1)DD$. $\ \ \ \ \ \ \ \ \ \ \ \ \ \ \ \ \
\ \ \ \ \ \ \ \ \ \ \ \ \ \ \ \ \ \ \ \ \ \ \ \ \ \ \ \ \ \ \ \ \ \
\ \ \ \ \ \ \ \ \ \ \ \ \ \ \ \ \ \ \ \ \ \ \ \ \ \ \ \ \ \ \ \ \ \
\ \ \ \ \ \ \ \ \ \ \ \ \ \ \ \ \ \ \ \ \ \ \ \ \ \ \ \ \ \ \ \ \ \
\ \ \ \ \ \ \ \ \ \ \ \ \ \ \ \ \ \ \ \ \ \ \ \ \ \ \ \ \ \ \ \ \ \
\ \ \ \ \ \ \ \ \ \ \ \ \ \ \ \ \ \ \ \ \ \ \ \ \ \ \ \ $

\noindent$\bullet\ \ \ v=67$: begin with a $4-GDD (6^{4}9^{1})$ $\cite{mullin}$.
Applying lemma 2.2 using a weight 2, we obtain a super-simple
$2-(67,4,1)DD$ with $f\geq\frac{1}{2}$. }
\end{proof}

\section{Asymptotic results }  
In previous section, we showed that for all admissible values of
$v$, except $v=7$, there exists a super-simple $2-(v,4,1)DD$ with
$f\geq\frac{1}{2}$. In this section we show that these results can
be improved for $2-(v,4,1)DD$s. And  we prove that for all
sufficiently large admissible $v$, there exists a $2-(v,4,1)DD$ with
$f\geq\frac{5}{8}-\frac{c}{v}$ for suitable positive constant c. For
this  result we need a directed group divisible design with block
size four $(4-DGDD)$ of type $(2^4)$ with $f\geq\frac{5}{8}$.

A $4-DGDD$ of type $(2^4)$ can be constructed with groups $\{1,2\},\
\{3,4\},\ \\ \{5,6\},\ \{7,0\}$, and with the following blocks.

$$
\begin{array}{llll}
blocks: &5\ 1\ 0\ 3\ \ \  & 3\ 7\ 1\ 5\ \ \  & 4\ 0\ 2\ 5  \\
        &6\ 0\ 1\ 4\ \ \  & 4\ 1\ 7\ 6\ \ \  & 5\ 2\ 7\ 4  \\
        &2\ 3\ 0\ 6\ \ \  & 6\ 7\ 3\ 2\ \ \  &
\end{array}
$$
Each of first two columns above is a cyclical trade of volume 3,
hence any defining set of this directed group divisible design must
contain at least two 4-tuples from each of the first two columns.
The 4-tuples in the last column above form a trade and one of them
must be in any defining set. Therefore any defining set of this
$4-DGDD$ must contain at least 5 blocks, so it has $f\geq\frac{5}{8}$.

\begin{lemma}
For all $v\equiv1\ (mod\ 12)$, $v\geq61$, there exists a
$2-(v,4,1)DD$ with $f\geq\frac{5}{8}-\frac{13}{8v}$.
\end{lemma}
\begin{proof}
{For all $k>4$ there exists a $4-GDD (6^k)$ $\cite{mullin}$. Let
$(G,\mathcal{B})$ be a group divisible design with element set U,
blocks of size 4 and groups of size 6. We replace each group $g\in
G$ with a $2-(13,4,1)DD$ with $f\geq\frac{1}{2}$ on $g\times
Z_2\bigcup\{\infty\}$ and each block $b\in \mathcal{B}$ with $4-DGDD$
of type $(2^4)$ with $f\geq\frac{5}{8}$ on $b\times Z_2$, such that
its groups are $\{x\}\times Z_2$, $\{y\}\times Z_2$, $\{z\}\times
Z_2$ and $\{w\}\times Z_2$. Since a $4-GDD (6^k)$ has $3k(k-1)$
blocks, so a $2-(v,4,1)DD$ has
\begin{center}
 $\large f\geq(\frac{5}{8}.24k(k-1)+\frac{1}{2}.26k)/(\large\frac{(12k+1)(12k)}{6}).$
\end{center}
Simplifying gives $\large f\geq\frac{5}{8}-\frac{13}{8v}$. }
\end{proof}
\begin{lemma}
For all $v\equiv4\ (mod\ 12)$, $v\geq64$, there exists a
$2-(v,4,1)DD$ with $f\geq\frac{5}{8}-\frac{1}{2v}$.
\end{lemma}
\begin{proof}
{For all $k>4$ there exists a $4-GDD (2^{3k+1})$ $\cite{mullin}$.
Let $(G,\mathcal{B})$ be a group divisible design with element set
U, blocks of size 4 and groups of size 2. We replace each group
$g\in G$ with a $2-(4,4,1)DD$ on $g\times Z_2$ and each block $b\in
\mathcal{B}$ with $4-DGDD$ of type $(2^4)$ with $f\geq\frac{5}{8}$ on
$b\times Z_2$, such that its groups are $\{x\}\times Z_2$,
$\{y\}\times Z_2$, $\{z\}\times Z_2$ and $\{w\}\times Z_2$. Since a
$4-GDD (2^{3k+1})$ has $k(3k+1)$ blocks, so a $2-(v,4,1) DD$ has
\begin{center}
 $\large f\geq(\frac{5}{8}.8k(3k+1)+\frac{1}{2}.4(3k+1))/(\frac{(12k+4)(12k+3)}{6}).$
\end{center}
Simplifying gives $\large
f\geq\frac{5}{8}-\frac{1}{8(4k+1)}>\frac{5}{8}-\frac{1}{2v}$. }
\end{proof}
\begin{lemma}
For all $v\equiv7\ (mod\ 12)$, $v\geq67$, there exists a
$2-(v,4,1)DD$ with
$f\geq\frac{5}{8}-\frac{13}{8v}-\frac{21}{2v(v+11)}$.
\end{lemma}
\begin{proof}
{For all $k>4$ there exists a $4-GDD (6^k 9^1)$ $\cite{gennian}$.
Let $(G,\mathcal{B})$ be a group divisible design with element set
U, blocks of size 4 and groups of size 6 and a single group of size
9. We replace each group $g\in G$ of size 6 and 9 with a
$2-(13,4,1)DD$ and a  $2-(19,4,1)DD$ with $f\geq\frac{1}{2}$ on
$g\times Z_2\bigcup\{\infty\}$, respectively, and each block $b\in
\mathcal{B}$ with $4-DGDD$ of type $(2^4)$ with $f\geq\frac{5}{8}$ on
$b\times Z_2$, such that its groups are $\{x\}\times Z_2$,
$\{y\}\times Z_2$, $\{z\}\times Z_2$ and $\{w\}\times Z_2$. Since a
$4-GDD (6^k 9^1)$ has $3k(k+2)$ blocks, so a  $2-(v,4,1)DD$ has
\begin{center}
 $\large f\geq(\frac{5}{8}.24k(k+2)+\frac{1}{2}.26k + 29)/(\frac{(12k+19)(12k+18)}{6}).$
\end{center}
Simplifying gives $\large
f\geq\frac{5}{8}-\frac{13}{8v}-\frac{21}{2v(v-1)}$. }
\end{proof}
\begin{lemma}
For all $v\equiv10\ (mod\ 12)$, $v\geq70$ there exists a
$2-(v,4,1)DD$ with $f\geq\frac{5}{8}-\frac{1}{4v}-\frac{1}{2v}$.
\end{lemma}
\begin{proof}
{For all $k>4$ there exists a $4-GDD (2^{3k} 5^1)$ $\cite{gennian}$.
Let $(G,\mathcal{B})$ be a group divisible design with element set
U, blocks of size 4 and groups of size 2 and   a single group of
size 5. We replace each group $g\in G$ of size 2 and 5 with a
$2-(4,4,1)DD$ and a $2-(10,4,1)DD$ with $f\geq\frac{1}{2}$ on
$g\times Z_2$, respectively, and each block $b\in \mathcal{B}$ with
$4-DGDD$ of type $(2^4)$ with $f\geq\frac{5}{8}$ on $b\times Z_2$,
such that its groups are $\{x\}\times Z_2$, $\{y\}\times Z_2$,
$\{z\}\times Z_2$ and $\{w\}\times Z_2$. Since a $4-GDD (2^{3k}
5^1)$ has $k(3k+4)$ blocks, so a $2-(v,4,1) DD$ has
\begin{center}
 $\large f\geq(\frac{5}{8}.8k(3k+4)+\frac{1}{2}.4(3k)+8)/ (\frac{(12k+10)(12k+9)}{6}).$
\end{center}
Simplifying gives $\large
f\geq\frac{5}{8}-\frac{1}{4v}-\frac{k+4}{2v(4k+3)}$ or $\large
f>\frac{5}{8}-\frac{1}{2v}$. }
\end{proof}

\begin{theorem}
For all $\epsilon > 0$ there exists $v_0(\epsilon)$ such that for
all admissible $v>v_0$ there exists a $2-(v,4,1)DD$ with
$f\geq\frac{5}{8}-\frac{c}{v}$ where $c=\frac{13}{8}+\epsilon$.
\end{theorem}
%


%

\end{document}